\documentclass[11pt, a4paper]{article}
\usepackage[a4paper, total={6.5in, 8.5in}]{geometry}
\usepackage[utf8]{inputenc} 
\usepackage[T1]{fontenc}
\usepackage[dvipsnames]{xcolor}
\usepackage{amsthm, amsmath, amssymb, amsfonts, graphicx, geometry, lipsum, amsxtra, array, enumitem, hyperref, mathrsfs, verbatim, epigraph, tabularx,mathtools, esint, tcolorbox, authblk}
\usepackage[symbol]{footmisc}

\numberwithin{equation}{section}

\usepackage[dvipsnames]{xcolor}
\usepackage{etoolbox}
\hypersetup{
    colorlinks=true,
    linkcolor=blue,
    citecolor=blue,
    urlcolor=red,
    pdftitle={[BKP23] Determination of lower order perturbations of a polyharmonic operator in two dimensions},
    }

\definecolor{titlepagecolor}{cmyk}{1,.60,0,.40}
\patchcmd{\subsection}{\normalfont}{\normalfont\color{blue}}{}{}
\DeclareFixedFont{\titlefont}{T1}{ppl}{b}{it}{0.5in}
\def\th@plain{%
  \thm@notefont{}
  \itshape 
}
\def\th@definition{%
  \thm@notefont{}
  \normalfont 
}
\makeatother

\usepackage{srcltx} 
\usepackage{amscd}
\usepackage{pb-diagram}
\usepackage[all]{xy}
\xyoption{matrix}
\xyoption{arrow}
\usepackage{pdfsync}
 
\usepackage{thmtools}

\DeclareUnicodeCharacter{2212}{-,+}
\usepackage[none]{hyphenat}

\theoremstyle{plain} 
\newtheorem{defn}{Definiton}[section]
\newtheorem{theorem}[defn]{Theorem}

\newtheorem{lemma}[defn]{Lemma}

\newtheorem{Theorem}{Theorem}

\title{Vanishing Elements of Prime Power Order}
\author{Sonakshee Arora\textsuperscript{\textdagger} and Rahul Dattatraya Kitture\textsuperscript{$\ddagger$}}
\date{}
\begin{document}
\maketitle 

\begin{abstract}
\noindent An element $x$ in a finite group $G$ is said to be \textit{vanishing} if some (complex) irreducible character of $G$ takes value $0$ at $x$. In this article, we prove that every non-abelian  finite simple group, except $\mathrm{SL}_2(4)$ and $\mathrm{SL}_2(8)$, contains a vanishing element \textit{of prime power order} whose conjugacy class size is divisible by three  distinct primes. 
We use this result to obtain the following generalization of a result of Robati ($2021$): If $G$ is a non-solvable finite group in which, the conjugacy class size of all the vanishing elements of prime power order has at most two distinct prime divisors, then $G/\mathrm{Sol}(G)$ is a direct product of mutually isomorphic simple groups among $\mathrm{SL}_2(4)$ and $\mathrm{SL}_2(8)$. ($\mathrm{Sol}(G)$ is the largest normal solvable subgroup of $G$.) 
\vskip3mm
\noindent{\bf Keywords.} Vanishing elements, irreducible characters, simple groups, conjugacy classes
		
\noindent{\bf Mathematics Subject Classification (2020)}: 20C15, 20E45, 20E32
\end{abstract}
    
\section{Introduction}
Investigating the structure of a group subject to some appropriate conditions on conjugacy classes has remained a problem of interest for many mathematicians (see the survey  \cite{MR2875589}). This thread of research has been continued in recent years, by considering specific conjugacy classes of $G$, such as the conjugacy classes of \textit{vanishing elements} (see \cite{MR4546806}, \cite{MR3849585}, \cite{MR3839687}, \cite{MR3918622} and the references therein). Recall that an element $x$ in a group $G$ is said to be vanishing if there is a (complex) irreducible character of $G$ which takes value $0$ at $x$. By a classical result of Burnside ($1903$) (see \cite{MR1576762} or Theorem $3.15$, \cite{MR2270898}), every non-abelian group contains vanishing elements.  
In this article, we consider the vanishing elements of prime power order in non-solvable groups, and their influence on the structure of group. 
Our first main result is:

\vskip-5mm\noindent 
\begin{Theorem}\label{thma} Every non-abelian finite simple group $G$, other than $\mathrm{SL}_2(4)$ and $\mathrm{SL}_2(8)$, contains a vanishing element $x$ of some prime power order such that $|x^G|$ is divisible by three distinct primes.
\end{Theorem}

\noindent 
The conjugacy class sizes in $\mathrm{SL}_2(4)(\cong A_5)$ and $\mathrm{SL}_2(8)$ have at most two distinct prime divisors. It is interesting to note that, in the proof of Theorem \ref{thma} for alternating groups $A_n$, $n\ge 6$, we obtain $x\in A_n$ (of prime power order) such that $|x^{A_n}|$ is divisible by \textit{every} prime divisor of $|A_n|$.

\noindent 
Our second main theorem is motivated by the following main theorem of (\cite{MR4219877}): \textit{For every vanishing element $x$ in a non-solvable group $G$, if  $|x^G|$ is divisible by at most two distinct primes, then $G/\mathrm{Sol}(G)\cong \mathrm{SL}_2(4)$ or $\mathrm{SL}_2(8)$ (where $\mathrm{Sol}(G)$ is the largest normal solvable subgroup of $G$)}.

By considering vanishing elements of \textit{prime power order}, we obtain a generalization of this result,  as stated below.

\vskip-5mm\noindent 
\begin{Theorem}\label{thmb}
Let $G$ be a non-solvable group. Suppose, for every vanishing element $x\in G$ of prime power order, $|x^G|$ is divisible by at most two distinct primes. Then $G/\mathrm{Sol}(G)$ is isomorphic to $\mathrm{SL}_2(4) \times \cdots \times \mathrm{SL}_2(4)$ or $\mathrm{SL}_2(8)\times \cdots \times \mathrm{SL}_2(8)$ (finite direct product).
\end{Theorem}
\noindent 
As an example, let $G=S_3\times A_5$. Then elements of order $6$ in $G$ are vanishing, but their conjugacy class sizes are divisible by three distinct primes. 
However, the  vanishing elements of \textit{prime power order} in $G$ have conjugacy class size divisible by at most two distinct primes, hence Theorem \ref{thmb} is applicable to $G$.

\vskip2mm\noindent
It is worth mentioning that our proof relies on the classification of finite simple groups.

\vskip2mm
\noindent\textbf{Notations:}

\noindent All the groups considered here are finite groups. For $x,g\in G$, we write $x^g=gxg^{-1}$. The conjugacy class (resp. centralizer) of $x$ in $G$ is denoted by $x^G$ (resp. $C_G(x)$).
For a real number $c>0$, $[c]$ denotes the greatest integer $\le c$.  
If $H$ is a normal subgroup of $G$, then for $\psi\in \mathrm{Irr}(H)$ and $g\in G$, $\psi^g$ is the (irreducible) character of $H$ defined by $\psi^g(x)=\psi(gxg^{-1})$ for $x\in H$. For notations of simple groups, we refer to Wilson (\cite{MR2562037}, p.3). Other notations are standard (see \cite{MR2270898}, \cite{MR1357169}).
%
%

\section{Preliminaries:} 
In this section, we state some basic lemmas, which are used in the proof of main theorems of this article. The following lemma is well-known, and since its proof is elementary, we skip the proof. 
\begin{lemma}\label{4}
Let $G$ be a finite group and $N$ a normal subgroup of $G$. Then

$(i)$ For any $x\in N$, $|x^N|$ divides $|x^G|$; \hskip5mm $(ii)$ for any $g\in G$, $|(gN)^{G/N}|$ divides $|g^G|$. 
\end{lemma}
\begin{lemma}\label{3}(Lemma $5$, \cite{MR2262862})
Let $G$ be a finite group and $M=S_1\times \cdots \times S_k$ be a minimal normal subgroup of $G$, where each $S_i$ is isomorphic to a non-abelian simple group $S$. If $\chi\in \mathrm{Irr}(S)$ extends to $\mathrm{Aut}(S)$ then $\chi\times \cdots \times \chi\in \mathrm{Irr}(M)$ extends to $G$.   
\end{lemma}

\noindent The following arithmetic results are useful to find some nice conjugacy classes in group $A_n$ for $n\ge 6$. 
\begin{theorem}\label{8}(\cite{MR0050615}, Nagura)
For any $x\ge 25$, there is a prime between $x$ and $(6/5)x$. 
\end{theorem}

\begin{lemma}\label{9}
Fix an integer $n\geq 9$. If $p$ is the largest (odd) prime such that $p^2\le n$, then $p^3>n$ and $[n/p^2]\in \{1,2\}$.
\end{lemma}
\begin{proof} For $9\le n \le 25$, the lemma can be easily verified. Assume $n\ge 25$; then $p\ge 5$. 

If possible, let $p^3\le n$. By Bertrand's postulate (see Theorem 8.7, \cite{MR1083765}), there is a prime $p_1$ with $p<p_1<2p$. 
Then $p^2<p_1^2< 4p^2 < p^3 \le n$, so the prime $p_1$ contradicts the hypothesis on $p$. 

To prove that $[n/p^2]\in \{1,2\}$, by division algorithm, we write $n=kp^2+r$ with $0\le r<p^2$ and show that $k\in \{1,2\}$. Suppose, if possible $k\ge 3$. Then $3p^2\le n$. 

For $p\ge 25$, by Theorem \ref{8}, there is a prime $p_2$ such that $p<p_2<\frac{6}{5}p<\sqrt{3}p$. Hence $p^2<p_2^2<3p^2\le n$, so the prime $p_2$ contradicts the hypothesis on $p$. 

For $p<25$, an easy computation shows that there is a prime $p_3$ with $p<p_3<\sqrt{3}p$, and by above arguments, the prime $p_3$ will contradict the hypothesis on $p$.
\end{proof}
\section{Elements of prime power order in simple groups}

The results of this section gives a fundamental setup for the proof of main theorems in this article.
\begin{theorem}\label{theorem2.3}
If $G$ is a non-abelian simple group, not isomorphic to $\mathrm{SL}_2(4)$ or $\mathrm{SL}_2(8)$, then there exists $x\in G$ of prime power order such that $|x^G|$ is divisible by three distinct primes. 
\end{theorem}

\noindent\textbf{Remark:} If $G$ is a non-abelian simple group (other than $\mathrm{SL}_2(4)$, $\mathrm{SL}_2(8)$), then Theorem 2.8 in \cite{MR4219877} shows the existence of $x\in G$ such that $|x^G|$ is divisible by three distinct primes;  however, the element $x$ (in the proof of cited result) is not necessarily of prime power order, specifically in $A_n$ and in classical simple groups.

\begin{proof}  We consider  the broad classification of (non-abelian) finite simple groups into four families. 

\vskip2mm\noindent 
\textbf{I. Alternating groups $A_n$:}
\noindent 
Since $A_5\cong \mathrm{SL}_2(4)\not\cong G$ (by hypothesis), we must have $n\ge 6$.  For $n\in \{6,7,8\}$, let $x=(a,b,c,d)(e,f)\in A_n$; then   $|x^{A_n}|$ is divisible by every prime divisor of $|A_n|$.

\noindent Let $n\ge 9$ and $p$ be the largest (odd) prime such that $p^2 \le n$. By Lemma \ref{9}, $n=kp^2+r$ where $k\in \{1,2\}$ (and $0\le r<p^2$). If $k=1$, we take $x\in A_n$ to be a cycle of length $p^2$, and if $k=2$, we take $x\in A_n$ to be a product of two disjoint cycles of length $p^2$. Then it is easy to see that 
 \begin{equation}\label{3.1}
     |x^{A_n}|=\frac{n!}{p^{2k}k!r!} \hskip5mm \mbox{ or } \hskip5mm |x^{A_n}|=\frac{n!}{2p^{2k}k!r!}.
 \end{equation}

\noindent\textbf{Claim:} Every prime divisor $l$ of $|A_n|$ divides $|x^{A_n}|$.

\vskip1mm\noindent 
\textbf{Case 1.} $l=p$

\noindent If $p\nmid |x^{A_n}|$, then $x$ centralizes a Sylow $p$-subgroup of $A_n$, say $P$, i.e. $x\in Z(P)$. Then $Z(P)$ will be of exponent $\ge p^2$. But, from the structure of Sylow $p$-subgroup of $A_n$ (see Theorem 1.6.19, \cite{MR1357169}), $Z(P)$ has exponent $p$, a contradiction. Thus, $p$ must divide $|x^{A_n}|$.

\noindent \textbf{Case 2.}: $l=2$.

\noindent Since $p^2\ge 9$ and $k\in \{1,2\}$, so $n=kp^2+r = k+r+k(p^2-1)\ge k+r+8$, hence $n!/(k!r!)=[S_n:S_k\times S_r]$ is divisible by $8!$. So from (\ref{3.1}), $|x^{A_n}|$ is divisible by $2$ (as $p>2$).

\vskip1mm\noindent 
\textbf{Case 3.}  $l$ is prime other than $2$ and $p$:

\noindent Since $k\in \{1,2\}$, so $(l, 2p^{2k}k!)=1$. If $(l,r!)=1$ then by (\ref{3.1}), $l$ divides $|x^{A_n}|$. 

\noindent Suppose $l$ divides $r!$; we show that $l$ divides $n!/r!$. 

\noindent Since $r+r<r+p^2 \le n$, so $|A_r \times A_r|$ divides $|A_n|$, i.e. $\frac{n!}{2}=\frac{r!}{2} \frac{r!}{2}s$ for $s\ge 1$. Then $\frac{n!}{r!}=\frac{r!}{2}s$, which is divisible by $l$ (since $(l,2)=1$ and $l$ divides $r!$)). 

\noindent Thus for $n\ge 6$, there is $x\in A_n$ of prime power order with $|x^{A_n}|$ divisible by every prime dividing $|A_n|$. 

\vskip3mm
\noindent \textbf{II. Sporadic and Exceptional Simple Groups of Lie Type:} 

\noindent Following Atlas \cite{MR0827219} (or \cite{MR2562037}), there are $27$ sporadic simple groups (including Tits group), and in each of these groups (again using Atlas \cite{MR0827219}), we can verify that there is an element of order $2$ whose conjugacy class size is divisible by three distinct primes. (In the table below, take $q$ as prime power unless stated otherwise.) 
 \begin{center}
     \begin{tabular}{|c| c| c| c|}
     \hline
        $S$ & Condition & $|T|$ & Some divisors of $|t^S|$  \\\hline 
         $^2B_2(q)$ & $q=2^{2k+1}\ge 8$ & $q+\sqrt{2q}+1$ & $q(q-1)(q-\sqrt{2q}+1)$\\  
         $^3D_4(q)$ &  & $q^4-q^2+1$ & $q(q^6-1)$\\
         $G_2(q)$ & $q\not\equiv 1(3),\,\, q\ge 3$ & $q^2+q+1$ & $q(q^2-1)(q^2-q+1)$\\
                  & $q\not\equiv 2(3),\,\, q\ge 3$ & $q^2-q+1$ & $q(q^2-1)(q^2+q+1)$\\
         $^2G_2(q)$ & $q={3^2k+1}\ge 27$ & $q+\sqrt{3q}+1$ & $q(q^2-1)$\\
         $F_4(q)$ &  & $q^4-q^2+1$ & $q(q^8-1)$\\
         $^2F_4(q)$ & $q=2^{2k+1}\ge8$ & $q^2+\sqrt{2q^3}+q+\sqrt{2q}+1$ & $q(q^4-1)$\\
         $E_6(q)$ &  & $\frac{q^6+q^3+1}{(3,q-1)}$ & $q(q^6-1)$\\
         $^2E_6(q)$ & & $\frac{q^6-q^3+1}{(3,q+1)}$ & $q(q^6-1)$\\
         $E_8(q)$ &  & $q^8-q^4+1$ & $q(q^8-1)$\\
         $^2F_4(2)'$ &  & $13$ & $2.3.5$\\\hline
     \end{tabular}
 \end{center}

\noindent Let $S$ be an exceptional simple group of Lie type other than $E_7(q)$. By Theorem 3.1 in (\cite{MR2507573}) and its proof,  $S$ contains a cyclic maximal torus $T$ of order given in the table above, and for the pair $(S,T)$, we choose any $1\neq t\in T$ of prime power order. Then (again by Theorem 3.1, \cite{MR2507573}) $C_S(t)=T$. 

 For $m\ge 3$, $\mathrm{gcd}(m,m+2)\le 2$, it follows that  $m(m+1)(m+2)$ is divisible by three distinct primes. This can be used to verify that the divisors of $|t^S|$ in the last column of table above are divisible by three distinct primes (since in almost all the cases, $(q-1)q(q+1)$ divides $|t^S|$). For $q(q-1)(q-\sqrt{2q}+1)$ with $q=2^{2k+1}\ge 8$, observe that $q-1$ and $q-\sqrt{2q}+1$ are coprime. By (Lemma 2.7, Chapter IX, \cite{MR0650245}) $q^4-1$ is divisible by two distinct primes.

Let $S=E_7(q)$. From the proof of Theorem 3.4 in \cite{MR2507573}, there is a maximal torus $T_1$ in $S$ of order   $\frac{q^7-1}{(2,q-1)}$. Take $t\in T_1$ whose order is a primitive prime divisor of $q^7-1$. Then $C_{S}(t)=T_1$, and so $|t^S|=[S:T_1]$ is divisible by  $q(q^6-1)$, which is divisible by three distinct primes. 

%
%
\vskip5mm\noindent 
\textbf{III. Classical Simple Groups:}

\noindent For all classical finite simple groups $S$, we explicitly describe an element $t\in S$ of prime power order such that $|t^S|$ is divisible by three distinct primes. 

The table below considers the classical simple groups $S$ of degree $\ge 4$, except $\mathrm{P}\Omega_n^+(q)$ ($n\ge 8$ even); the remaining classical groups are considered separately. (Recall: a prime $p$ is called a Zsigmondy prime for $\langle a,n\rangle$ if $p\nmid a$ and order of $a \pmod{p}$ is $n$, see \cite{MR1402885}.) From the conditions on $q$ and $n$ for each $S$, we choose a Zsigmondy prime $p$ for a pair $\langle q,m\rangle$ as listed in the third column except when $\langle q,m\rangle=\langle 2,6\rangle$ (existence of such $p$ follows by Theorem 3 in \cite{MR1402885} since $m\ge 4$ for all the cases in the table;  further, $p$ will be odd). We take $p=3$ if $\langle q,m\rangle=\langle 2,6\rangle$. 

\vskip2mm
Let $G$ be the (classical) quasi-simple group corresponding to $S$ (i.e. $G/Z(G)=S$). Then there is a regular semisimple prime-power order element $\tilde{t}\in G$
(such that $\tilde{t}Z(G)=t$), whose diagonalization is given in second column. Here, 
$\lambda\in\overline{\mathbb{F}}_q$ is a generator of Sylow $p$-subgroup of $\mathbb{F}_{q^{2n}}^*$ (resp. of $\mathbb{F}_{q^{n}}^*$) if $S=\mathrm{PSU}_n(q)$ (resp. if $S\neq \mathrm{PSU}_n(q)$).

\begin{center}
\small
\begin{tabular}{|c|c|c|c|c|}\hline
$S$ &  $\tilde{t}\in G$ diagonalizes  & $p$ is Zsigmondy  & Some divisors \\
 &   (over $\overline{\mathbb{F}}_q$) to & prime for & of $|t^S|$ \\ \hline\hline
$\mathrm{PSL}_n(q)$ ($n\ge 4$) 
& $\mbox{diag} (\lambda,\lambda^q,\cdots,\lambda^{q^{n-1}})$,  
&  $\langle q,n\rangle$ 
&  $q(q^2-1)(q^3-1)$ \\
\hline
$\mathrm{PSU}_n(q)$ ($n\ge 4$ even) 
& $\mbox{diag} (\lambda,\lambda^{q^2},\cdots,\lambda^{q^{2(n-2)}},1)$  
& $\langle q,2(n-1)\rangle$  
&  $q(q^4-1)$\\
\hline 
$\mathrm{PSU}_n(q)$ ($n\ge 5$ odd) 
& $\mbox{diag} (\lambda,\lambda^{q^2},\cdots,\lambda^{q^{2(n-1)}})$  
& $\langle q,2n\rangle$   
& $\displaystyle q(q^4-1)$  \\\hline
$\mathrm{PSp}_n(q)$ ($n\ge 6$ even)
& $\mbox{diag} (\lambda,\lambda^{q},\cdots,\lambda^{q^{(n-1)}})$  
&  $\langle q,n\rangle$  
& $\displaystyle q(q^4-1)$ \\
\hline 
$\mathrm{P\Omega}_{n}(q)$ ($n\ge 7$ odd) 
& $\mbox{diag} (1,\lambda,\lambda^{q},\cdots,\lambda^{q^{n-2}})$  
& $\langle q,n-1\rangle$  
& $\displaystyle  q(q^4-1)$  \\
{\tiny ($q$ must be odd)} &  &  &  \\\hline
$\mathrm{P\Omega}_{n}^-(q)$  ($n\ge 8$ even)
& $\mbox{diag} (\lambda,\lambda^{q},\cdots,\lambda^{q^{n-1}})$  
&  $\langle q,n\rangle$  
&  $\displaystyle q(q^4-1)$\\\hline 
\end{tabular}
\end{center}
\normalsize

\vskip2mm\noindent 
Observe that $|\tilde{t}^G|=|t^S|$: the choice of Zsigmondy prime $p$ implies that $p$ is coprime to $|Z(G)|$, for each $G$ corresponding to $S$ in the table. On the other hand, $o(t)$ and $o(\tilde{t})$ are powers of $p$. From this, it is easy to see that $|\tilde{t}^G|=|t^S|$.

We verify below that the divisors of $|t^S|$, i.e. of $|\tilde{t}^G|$, are as shown in the last column: By Lemma 4.6 in \cite{MR2507573}, since $\tilde{t}\in T$ is regular semisimple, we get $C_G(\tilde{t})=T$, where $T$ is a maximal tori in $G$ containing $\tilde{t}$, which we choose from table III in \cite{MR2507573} as:
$$
T:=T_1 \mbox{ if } G\neq \mathrm{SU}_n(q), n\ge 4 \mbox{ even} \hskip10mm \mbox{ and } \hskip10mm T:=T_2 \mbox{ if }  G=\mathrm{SU}_n(q), n\ge 4 \mbox{ even}.
$$
 The orders of $G$ are well-known and the orders of $T$ can be found in  table III of \cite{MR2507573}.
Now, we can easily verify that all terms in the last column of above table are divisible by three distinct primes:

$\bullet$ $q(q^2-1)(q^3-1)$ has three coprime factors namely $q,q+1,q^2+q+1$.

$\bullet$ By (Lemma 2.7(c), Chapter IX, \cite{MR0650245}) $q^4-1$ is divisible by two distinct primes.

\vskip2mm
Next, consider the simple groups $S=\mathrm{P}\Omega_n^+(q)$ ($n\ge 8$ even) and let $G=\Omega_n^+(q)$ be the quasisimple group so that $G/Z(G)\cong S$. Then $G$ contains a semisimple element $\tilde{t}$, which diagonalizes over $\overline{\mathbb{F}}_q$ as
$$
\mathrm{diag}(\lambda,\lambda^q,\cdots,\lambda^{q^{\frac{n}{2}-2}}, \lambda^{-1},\lambda^{-q},\cdots, \lambda^{-q^{\frac{n}{2}-2}},1,1)=:\mbox{diag}(\mathbf{t},1,1).
$$
Here $[\mathbb{F}_q(\lambda):\mathbb{F}_q]=n-2$, and order of $\lambda$ is the order of Sylow $p$-subgroup of $\mathbb{F}_{q^{n-2}}^*$, where $p$ is (odd) Zsigmondy prime for $\langle q,n-2\rangle\neq \langle 2,6\rangle$; we take $p=3$ if  $\langle q,n-2\rangle = \langle 2,6\rangle$

As mentioned before, since $|\tilde{t}^G|=|t^S|$, we show that $q(q^4-1)$ divides $|\tilde{t}^G|$. 
We write  $V=V_1\perp V_2$ (orthogonal sum) where $V_2$ is the eigenspace of $\tilde{t}$ with eigenvalue $1$. Then 
$$
C_G(\tilde{t})=C_{\Omega_n^+(q)}(\tilde{t})=C_{\Omega_{n-2}^+(q)}(\mathbf{t})\times \Omega_2^+(q)= \Big{(}C_O^2(n-2,q) \cap \Omega_{n-2}^+(q) \Big{)} \times \Omega_2^+(q)  
$$
where $C_O^2(n-2,q) \le O^+_{n-2}(q)$ is a cyclic subgroup of order $q^{\frac{n-2}{2}}-1$  (see p.$101$ \cite{MR2507573} ). For last equality above, we note that: (i) $\mathbf{t}\in\Omega_{n-2}^+(q)$ is regular semisimple, so $C_{\Omega_{n-2}^+(q)}(\mathbf{t})$ is  abelian, and (ii) from the proof of Lemma 4.6 in \cite{MR2507573} , $C_O^2(n-2,q)$ is maximal abelian (and $C_O^2(n-2,q)\cap \Omega_{n-2}^+(q)\subseteq C_{\Omega_{n-2}^+(q)}(\mathbf{t})$). 

Observe that 
\begin{equation*}
G=\Omega_{n}^+(q) \supseteq \Omega_{n-2}^+(q)\times \Omega_2^+(q)  \supseteq \Big{(}C_O^2(n-2,q) \cap \Omega_{n-2}^+(q) \Big{)} \times \Omega_2^+(q) =C_G(\tilde{t}).
\end{equation*}

Then $|t^S|=|\tilde{t}^G|=[G:C_G(\tilde{t})]$, which is divisible by 
$$
\frac{|\Omega_{n-2}^+(q)|}{|C_O^2(n-2,q) \cap \Omega_{n-2}^+(q)|}, \mbox{ which is divisible by } 
\frac{|\Omega_{n-2}^+(q)|}{|C_O^2(n-2,q)|} = \frac{|\Omega_{n-2}^+(q)|}{|q^{(n-2)/2}-1|}.
$$
Since $n\ge 8$ (even), from the order of $\Omega_{n-2}^+(q)$, we can see that $|t^S|$ is divisible by $q(q^4-1)$. (Before starting the proof for $\mathrm{P}\Omega_n^+(q)$, we have seen that $q(q^4-1)$ is divisible by three distinct primes. Also, $t$ is of order a power of $p$, with $p$ an odd prime.)

\vskip2mm\noindent 
Finally, consider the remaining (classical) simple groups $S$ (namely, those of degree $<4$).

For $S=\mathrm{PSL}_2(q)$, if $q-1$ is divisible by two distinct primes, we take $\tilde{t}\in \mathrm{SL}_2(q)$, which diagonalizes (over $\overline{\mathbb{F}}_q)$) to $\mbox{diag}(\lambda,\lambda^q)$ where $\lambda \in \mathbb{F}_{q^2}\setminus \mathbb{F}_q$ with $\lambda^{1+q}=1$. Let $t\in S$ be the image of $\tilde{t}$. Then  $|\tilde{t}^{\mathrm{SL}_2(q)}|= |t^S|=q(q-1)$, and is divisible by three distinct primes. If $q-1$ is power of a prime, say $p^a$, then by (Lemma 2.7, Chapter IX, \cite{MR0650245}), one of the following holds: 
\begin{center}
\textbf{(a)} $q=2^l$ such that $2^l-1$ is a Mersenne prime; \,\,\,\textbf{(b)} $q$ is a Fermat prime; \,\,\, \textbf{(c)} $q=9$. 
\end{center}
\noindent In case \textbf{(a)}, $S = \mathrm{PSL}_2(2^l)$.  Since $S\not\cong  \mathrm{PSL}_2(4)$ or $\mathrm{PSL}_2(8)$, so $l\ge 4$. Since $2^l-1$ is a prime, so $l$ must be odd (hence $\ge 5$). 
Let $t\in S$ be the image of $\begin{pmatrix} 1 & 1 \\ 0 & 1 \end{pmatrix}\in \mathrm{SL}_2(q)$. Then  $|t^S|= 2^{2l}-1=(2^l-1)(2^l+1)$. Since $l\ge 5$ is odd, by (Lemma 2.7, Chapter IX, \cite{MR0650245}) one can verify that $2^l+1$ is divisible by two distinct primes (and since $2^l-1$ is a prime), it follows that $|t^S|$ is divisible by three distinct primes. 

\vskip2mm\noindent 
In case \textbf{(b)}, $q$ is a Fermat prime, say $q=2^{2^l}+1$. Consider $t=\begin{pmatrix} 1 & 1 \\ 0 & 1 \end{pmatrix}\in \mathrm{SL}_2(q)$, its image in $S=\mbox{PSL}_2(q)$ has order $q$ and $|t^S|=(q^2-1)/2 =(2^{2^l})(2^{2^l-1}+1)$. If $l>2$, then by (Lemma 2.7(c), Chapter IX, \cite{MR0650245}), one can see that $2^{2^l-1}+1$ is divisible by two distinct (odd) primes, hence $|t^S|$ is divisible by three distinct primes. Now, if $l=2$, then $S\cong \mathrm{PSL}_2(2^{2^l}+1)=\mathrm{PSL}_2(17)$. This group has an element of order $4$, with conjugacy class size  is divisible by three distinct primes ($2,3,17$).

\vskip2mm
\noindent In case \textbf{(c)}, $q=9$. Then $S=\mathrm{PSL}_2(9)\cong A_6$; in this permutation group, the element $(1,2,3,4)(5,6)$ (of order $4$)) has conjugacy class size $2\cdot 3^2\cdot 5$.

\vskip2mm\noindent 
For $S=\mathrm{PSL}_3(q)$, if $q\ge 4$, let $p$ be a Zsigmondy prime for $\langle q,3\rangle$; it will be odd. Take $\lambda\in\overline{\mathbb{F}}_q$ of order a power of $p$ such that $[\mathbb{F}_{q}(\lambda):\mathbb{F}_{q}]=3$. Then take $\tilde{t}\in \mathrm{SL}_3(q)$, which diagonalizes (over $\overline{\mathbb{F}}_q)$) to $\mbox{diag}(\lambda,\lambda^q,\lambda^{q^2})$, and $t\in S$ be the image of $\tilde{t}$. Then order of $t$ is a power of $p$ and  $|t^S|=q^3(q^2-1) = q^2(q-1)q(q+1)$ and $(q-1)q(q+1)$ is divisible by three distinct primes. If $q=3$, take $\tilde{t}$ similar to $\mbox{diag}(1,\mu^2,\mu^{-2})$ where $\mu\in \mathbb{F}_9^*$ is of order $8$.  Then $t$ is conjugate to an element of order $4$ in $\mbox{PSL}_3(3)$ whose conjugacy class size is $2\cdot 3^3\cdot 13$. If $q=2$, $\mbox{PSL}_3(q)$ is not simple.

\vskip2mm\noindent
For $S=\mbox{PSU}_3(q)$, if $q\ge 4$, let $p$ be a Zsigmondy prime for $\langle q,6\rangle$; it will be odd. Let $\lambda\in\overline{\mathbb{F}}_{q^2}$ be of order a power of $p$ with $[\mathbb{F}_{q^2}(\lambda):\mathbb{F}_{q^2}]=3$ and consider $\tilde{t}\in \mathrm{SU}_3(q)$, which diagonalizes (over $\overline{\mathbb{F}}_q)$) to $\mbox{diag}(\lambda,\lambda^{q^2},\lambda^{q^4})$, and let $t\in S$ be the image of $\tilde{t}$. Then  $|t^S|=q^3(q^2-1)$ and since $q\ge 4$, so $(q-1)q(q+1)$ is divisible by three distinct primes,  so is $|t^S|$. If $q=3$, take take $\tilde{t}$ similar to $\mbox{diag}(1,\mu^2, \mu^{-2})$, where $\mu$ generates $\mathbb{F}_9^*$. Then $t$ has order $4$ and its conjugacy class size in $\mbox{PSU}_{3}(3)$ is $3^3\cdot2\cdot 7$. (Note that $\mathrm{PSU}_3(2)$ is not simple.)

\vskip2mm\noindent
For $S=\mbox{PSp}_4(q)$, if $q\ge 4$, let  $p$ be a Zsigmondy prime for $\langle q,4\rangle$; it will be odd. Let $\lambda\in\overline{\mathbb{F}}_{q}$ be of order a power of $p$ such that $[\mathbb{F}_q(\lambda):\mathbb{F}_q]=4$. Consider $\tilde{t}\in \mathrm{Sp}_4(q)$, which diagonalizes (over $\overline{\mathbb{F}}_q)$) to $\mbox{diag}(\lambda,\lambda^{q},\lambda^{q^2},\lambda^{q^3})$ and let $t\in S$ be its image. Then $|t^S|=q^4(q^2-1)$. Since $q\ge 4$, $(q-1)q(q+1)$ is divisible by three distinct primes, and so is $|t^S|$. If $q=3$, then with $J=\begin{bmatrix}0 & -1\\1 & 0\end{bmatrix}$, the image of $\begin{bmatrix} J & \mathbf{0}\\\mathbf{0}& I_2\end{bmatrix}$ in $\mbox{PSp}_4(3)$ has order $2$, with conjugacy class size $2\cdot 3^3\cdot 5$. (Note that $\mathrm{PSp}_4(2)\cong S_6$ is not simple.)
\end{proof}
%
%
%

\noindent The following theorem will be crucial for the proofs of Theorem \ref{thma} 
and Theorem \ref{thmc}. 

\begin{theorem}\label{theorem3.2}
Let $S$ be a finite simple group, which has no irreducible character of $q$-defect zero for some prime $q$. Then there exist $x\in S$ of prime power order such that $|x^S|$ is divisible by every prime dividing $|S|$ and there exist $\chi\in \mathrm{Irr}(S)$ which extends to $\mathrm{Aut}(S)$ with $\chi(x)=0$.
\end{theorem}

\begin{proof}
By (Corollary $2$, \cite{MR1321575}), the hypothesis on $S$ implies that $S$ is isomorphic to a sporadic group listed in the table below, or $S\cong A_n$ with $n\ge 7$. If $S$ is a sporadic group, the desired $x$ and $\chi$ are given in the table;  the notations are as in Atlas (see \cite{MR0827219}). 

\begin{center}
\begin{tabular}{|c|c|c|c|c|c|c|c|c|c|c|}\hline
$S\rightarrow $ & $M_{12}$ & $M_{22}$ & $M_{24}$  & $J_2$ & $HS$ & $Suz$ & $Ru$ & $Co_1$ & $Co_3$ & $BM$ \\\hline
$x$ & $3$B & $8$A & $4$C & $3$B & $5$C & $8$B & $4$D & $4$F & $4$B & $4$J \\\hline
$\chi$ & $\chi_7$ & $\chi_7$ & $\chi_7$ & $\chi_6$ & $\chi_7$ & $\chi_3$ & $\chi_{11}$ & $\chi_2$ & $\chi_6$ & $\chi_{20}$\\\hline
\end{tabular}
\end{center}
For $S=A_n$, $n\ge 7$, we obtain desired $x$ and $\chi$ as follows.  Recall that $x\in S_n$ has cycle type $(\lambda_1,\ldots,\lambda_k)$, if $x$ is a product of disjoint cycles of length $\lambda_i$ with $\lambda_1\ge \ldots \ge \lambda_k\ge 1$ and $\sum_i \lambda_i=n$; we call $(\lambda_1,\ldots,\lambda_k)$ a partition of $n$.  For each partition $\sigma$ of $n$, there is $\chi_{\sigma}\in \mathrm{Irr}(S_n)$ (see Theorem 2.4, \cite{MR0513828}) and its restriction to $A_n$ is irreducible if and only if the Young diagram corresponding to $\sigma$ is not symmetric (see Theorem $2.5.7$, \cite{MR0644144}).

 If $n=7$, let $x=(1,2,3,4)(5,6)$, $\sigma=(5,2)$ and $\chi_{\sigma}\in \mathrm{Irr}(S_7)$ the corresponding character. 

If $n=8$, let $x=(1,2,3,4)(5,6)$, $\sigma=(5,2,1)$ and $\chi_{\sigma}\in \mathrm{Irr}(S_8)$ the corresponding character. 

In both the cases, it can be easily seen that $|x^{A_n}|$ is divisible by every prime divisor of $|A_n|$, $\chi_{\sigma}$ is irreducible on $A_n$ and $\chi_{\sigma}(x)=0$ by Murnaghan-Nakayama formula (Theorem 2.4.7, \cite{MR0644144}).

For $n\ge 9$, the following table gives the cycle-type of $x$ and a partition $\sigma$ of $n$ which is not symmetric; and by Murnaghan-Nakayama formula (Theorem 2.4.7, \cite{MR0644144}) $\chi_{\sigma}(x)=0$. The proof of Theorem \ref{theorem2.3} shows that $|x^{A_n}|$ is divisible by every prime divisor of $|A_n|$. 
\begin{center}
\begin{tabular}{|c|c|c|}
\hline 
Condition on $n$ & Cycle type of $x$ & Partition $\sigma$ of $n$ \\\hline
$n=p^2$ & $(p^2)$ & $(p^2-2, 2)$\\ 
$n=p^2+r$ \,\, ($1\le r < p^2$) & $(p^2,1,\ldots,1)$ & $(r+1,1,\ldots, 1)$\\ 
$n=2p^2+r$ \,\, ($0\le r \le p^2-3$) & $(p^2,p^2,1,\ldots,1)$ & $(p^2-1,p^2-1,1,\ldots, 1)$\\ 
$n=2p^2+r$ \,\, ($r=p^2-2$) & $(p^2,p^2,1,\ldots,1)$ & $(p^2-2,p^2-2,p^2-2,1,1,1,1)$\\ 
$n=2p^2+r$ \,\, ($r=p^2-1$) & $(p^2,p^2,1,\ldots,1)$ & $(p^2-2,p^2-2,,1,1,1)$\\\hline 
\end{tabular}
\end{center}
\end{proof}

%
\section{Proof of Main Theorems:}

\noindent 
\textbf{Proof of Theorem \ref{thma}:} \noindent Let $S$ be a non-abelian  simple group, not isomorphic to $\mathrm{SL}_2(4)$ or $\mathrm{SL}_2(8)$. If $S$ has no irreducible character of $q$-defect $0$ for some prime $q$, then by Theorem \ref{theorem3.2}, $S$ contains a vanishing element $x$ of prime power order with $|x^S|$ divisible by three distinct primes.  If $S$ has an irreducible character of $q$-defect $0$ for all primes $q$, then by Theorem $8.17$ \cite{MR2270898}, all the non-identity elements in $S$ (and in particular of prime power order) are vanishing. By Theorem \ref{theorem2.3}, there exists $x\in S$ of prime power order with $|x^S|$ divisible by three distinct primes. \hfill$\square$

\vskip5mm
For the proof of Theorem \ref{thmb}, we first prove the following two results. 
\begin{lemma}\label{7}
 $G$ be a finite group with a proper minimal normal subgroup $N=S_1\times \cdots \times S_k$, where $k\ge 2$ and  $S_i$ is isomorphic to either $\mathrm{SL}_2(4)$ or $\mathrm{SL}_2(8)$ for all $i$. Then $G$ contains a vanishing element $x$ of prime power order such that $|x^G|$ is divisible by three distinct primes.   
\end{lemma}
\begin{proof} 

\vskip2mm\noindent
Let $\psi$ be an irreducible character of $N$ such that $\ker\psi=S_2\times \cdots \times S_k$ and 
$$
\psi_{S_1}=
\begin{cases}
\mbox{irreducible character of degree } 5  & \mbox{ if } S_1\cong \mathrm{SL}_2(4)\\    
\mbox{irreducible character of degree } 8 & \mbox{ if } S_1\cong \mathrm{SL}_2(8)
\end{cases}
$$
(Remark: In each case, there is a unique irreducible character of the mentioned degree, and it takes same value on the elements of same order in the corresponding group.) 

\vskip2mm\noindent 
\underline{Claim 1.} $I(\psi)$ is the normalizer of $S_2\times \cdots \times S_k$ in $G$ and is proper subgroup of $G$. 

\vskip2mm\noindent 
If $g\in I(\psi)$, then $g$ normalizes $\ker\psi=S_2\times \cdots \times S_k$. Conversely, let $g$ normalizes $S_2\times \cdots \times S_k$. Since $S_1$ is the only normal complement of $S_2\times \cdots \times S_k$ in $N$, so $g$  normalizes $S_1$. Then for any $(a_1,a_2,\ldots,a_k)\in S_1\times \cdots \times S_k$, we can write $(a_1,a_2,\ldots,a_k)^g$ as $(a_1^g,(a_2,\ldots,a_k)^g)$. Since $S_2\times \cdots \times S_k \subseteq \ker\psi$, we have 
$$
\psi^g(a_1,a_2\ldots,a_k)=\psi((a_1^g,(a_2\ldots,a_k)^g))=\psi_{S_1}(a_1^g) \hskip3mm \mbox{ and } \hskip3mm \psi(a_1,a_2,\ldots,a_k) =\psi_{S_1}(a_1).
$$
Since $a_1,a_1^g\in S_1$ have same order, from the remark after definition of $\psi$,  we get $\psi_{S_1}(a_1)=\psi_{S_1}(a_1^g)$. Hence 
$\psi^g(a_1,\ldots,a_k)=\psi(a_1,\ldots,a_k)$ for all $(a_1,\ldots,a_k)\in S_1\times \cdots \times S_k$. 

Hence $I(\psi)$ is equal to the normalizer of $S_2\times \cdots \times S_k$ in $G$, and it is proper subgroup of $G$ by minimality of the normal subgroup $N$. 
This proves the claim $1$. 

\vskip2mm\noindent 
\underline{Claim 2:} $G\setminus \cup_{g\in G} I(\psi)^g$ contains a vanishing element $x$ of prime-power order. 
\vskip2mm\noindent 
We have $I(\psi)<G$ and $G$ acts transitively on (left) coset-space $G/I(\psi)$; the union of stabilizers of the cosets is precisely $\cup_{g\in G} I(\psi)^g$. By a theorem of Fein, Kantor and Schacher (See Theorem 1, \cite{MR0636194}), $G$ contains a prime-power order element $x$ whose action on $G/I(\psi)$ is fixed-point-free, hence $x\notin \cup_{g\in G}I(\psi)^g$. 
By Clifford's theorem, there is $\theta\in \mathrm{Irr}(I(\psi))$ whose restriction to $N$ is $\psi+\cdots+\psi$ and $\theta^G$ is irreducible. Since $x\notin \cup_{g\in G}I(\psi)^g$, we get  $\theta^G(x)=0$. This proves the claim $2$. 

\vskip2mm\noindent 
\underline{Claim 3.} For the prime-power order element $x$ in Claim $2$, $|x^G|$ is divisible by $2\cdot 3\cdot 5$. 

\vskip2mm\noindent 
We have $x\in G\setminus \cup_{g\in G} I(\psi)^g$. If $xS_1x^{-1}=S_1$ then $x(S_2\times \cdots \times S_k)x^{-1}=S_2\times \cdots \times S_k$ (it is the only normal complement of $S_1$ in $N$), i.e. 
$x\in N_G(S_2\times \cdots \times S_k)=I(\psi)$ (by Claim 1), a contradiction. 

Without loss of generality, we can assume that $xS_1x^{-1}=S_2$. 

If $2\nmid |x^G|$, then $C_G(x)$ contains a Sylow $2$-subgroup of $G$, say $P_2$. Since $N\trianglelefteq G$, $N\cap P_2$ is a Sylow $2$-subgroup of $N=S_1\times \cdots \times S_k$, say $H_1\times \cdots \times H_k$. So, 
\begin{equation}\label{(ii)}
C_G(x)\supseteq H_1\times \cdots \times H_k.
\end{equation}
But, on the other hand, since $xS_1x^{-1}=S_2$, so $xS_1x^{-1}\cap S_1=\mathbf{1}$, and $xH_1x^{-1}\cap H_1=\mathbf{1}$; in particular, $x$ does not centralize \textit{any} element of $H_1$ (except $1$), contradicting (\ref{(ii)}). So $2$ divides $|x^G|$. 

In a similar way, we can show that $3$ and $5$ divide $|x^G|$. This proves claim $3$.

\noindent(Note that, since $N\cong \mathrm{SL}_2(4)$ or $\mathrm{SL}_2(8)$, the only prime divisors of $|N|$ are $2,3$ and $5$. If $q\neq 2,3,5$ is a prime which does not divide $|x^G|$, then $C_G(x)$ will contain a Sylow $q$-subgroup of $G$, say $P_q$; but $N\cap P_q$ will be trivial. Hence by above arguments, we can not conclude that $q$ divides $|x^G|$.)
\end{proof} 

\vskip-5mm\noindent 
\begin{theorem}\label{thmc}
Let $G$ be a non-solvable group. Suppose, for every vanishing element $x\in G$ of prime power order, $|x^G|$ is divisible by at most two distinct primes. Then 

(i) All  the minimal normal subgroups of $G/\mathrm{Sol}(G)$ are isomorphic.

(ii) Any minimal normal subgroup of $G/\mathrm{Sol}(G)$ is isomorphic to $\mathrm{SL}_2(4)$ or $\mathrm{SL}_2(8)$.
\end{theorem}

\begin{proof}
Assuming $\mathrm{Sol}(G)=1$, it suffices to show: (i) all minimal normal subgroups of $G$ are isomorphic; (ii) any minimal normal subgroup of $G$ is isomorphic to $\mathrm{SL}_2(4)$ or $\mathrm{SL}_2(8)$.  Let $N$ be any minimal normal subgroup of $G$. Then $N= S_1\times \cdots\times S_k$ (internal product), where each $S_i$ is isomorphic to a  simple group $S$, which is non-abelian (since $\mathrm{Sol}(G)=1$).

\vskip2mm\noindent 
\underline{Claim 1.} \textit{$N$ has an irreducible character of $q$-defect zero for every prime $q$ dividing $|N|$.}

\vskip1mm\noindent 
The prime divisors of $|N|$ are precisely the prime divisors of $|S_1|$. If possible, let $S_1$ has no irreducible character of $q$-defect for any prime $q$ dividing $|S_1|$.  By Theorem \ref{theorem3.2}, there exists  $x\in S_1$ of prime power order and $\chi\in\mathrm{Irr}(S_1)$ such that $|x^{S_1}|$ is divisible by three distinct primes, $\chi(x)=0$ and $\chi$ extends to $\mathrm{Aut}(S_1)$.  Hence by Lemma \ref{3}  $\chi\times \cdots \times\chi \in \mathrm{Irr}(N)$ extends irreducibly to $G$, it vanishes at $x$ and $|x^{S_1}|$ divides $|x^{G}|$ (by Lemma \ref{4}). Thus $x\in G$ is  vanishing, of prime power order, with $|x^G|$ divisible by three distinct primes, a contradiction. Thus for $S_1$ (and hence for $N)$, there is an irreducible character of $q$-defect zero for every prime $q$ dividing $|N|$. This proves claim $1$.

\vskip2mm\noindent 
\underline{Claim 2.} \textit{$N$ is isomorphic to $\mathrm{SL}_2(4)\times \cdots \times \mathrm{SL}_2(4)$ or $\mathrm{SL}_2(8)\times \cdots \times \mathrm{SL}_2(8)$.}

\noindent
By Claim 1 and by a theorem 
of Burnside (see Theorem 8.17, \cite{MR2270898}), every non-trivial element of $N$ is vanishing element of $G$. By hypothesis and by Lemma \ref{4}(1), for every element of $N$ of prime power order, its conjugacy class size in $N$ has at most two distinct prime divisors, hence this is true for each $S_i$.  By Theorem \ref{thma}, $S_i\cong \mathrm{SL}_2(4)$ or $S_i\cong \mathrm{SL}_2(8)$. This proves claim $2$. 

\vskip2mm\noindent 
\underline{Claim 3.} $N$ is isomorphic to $\mathrm{SL}_2(4)$ or $\mathrm{SL}_2(8)$. 

\noindent 
Otherwise, $N=S_1\times \cdots \times S_k$ with $k\ge 2$ (and each $S_i$ is isomorphic to either $\mathrm{SL}_2(4)$ or $\mathrm{SL}_2(8)$). By Lemma \ref{7}, $N$ has a vanishing element $x$ of prime power order whose conjugacy class size in $G$ is divisible by three distinct primes.  Now by Lemma \ref{4}, $x$ contradicts the hypothesis on $G$.  
Hence we proved that any minimal normal subgroup of $G$ is isomorphic to $\mathrm{SL}_2(4)$ or $\mathrm{SL}_2(8)$. 

This completes the proof of  theorem.
\end{proof}

\subsubsection*{Proof of Theorem \ref{thmb}:}
Without loss of generality, we assume that $\mathrm{Sol}(G)=1$. By Theorem \ref{thmc}, any minimal normal subgroup of $G$ is isomorphic to either $\mathrm{SL}_2(4)$ or $\mathrm{SL}_2(8)$ .
\vskip2mm\noindent 
\textbf{Case $1$.} $G$ has an unique minimal normal subgroup (say $N$). 

\noindent Then $N \le G  \le \mathrm{Aut}(N).$ Since $N\cong \mathrm{SL}_2(4)$ or $\mathrm{SL}_2(8)$, from the structure of automorphism groups of these groups, we get $[\mathrm{Aut}(N):N]\in \{2,3\}$. Thus, if $N\neq G$, then 
$$G\cong \mathrm{Aut(\mathrm{SL}_2(4))}\cong S_5 \hskip5mm \mbox{ or } \hskip5mm 
G\cong \mathrm{Aut(\mathrm{SL}_2(8))}\cong \mathrm{Ree}(3).$$
Since $S_5$ and $\mathrm{Ree}(3)$ contain a vanishing element of prime power order (namely, of order $4$ and $9$ respectively) whose cojugacy class size is divisible by three distinct primes ($2\cdot 3\cdot 5$ and $2\cdot 3\cdot 7$ respectively), a contradiction to hypothesis. Hence $G=N\cong \mathrm{SL}_2(4)$ or $\mathrm{SL}_2(8)$. 

\vskip2mm\noindent 
\textbf{Case $2$.} $G$ has two minimal normal subgroups, say $N_1\cong \mathrm{SL}_2(4)$ and $N_2\cong \mathrm{SL}_2(8)$:

\noindent Since $\mathrm{SL}_2(4)\times \mathrm{SL}_2(8)$ contains a vanishing element $(x,y)$ with $o(x)=o(y)=2$, and its conjugacy class size in $N_1\times N_2$ (hence in $G$) is divisible by $3\cdot 5\cdot 7$, this contradicts the hypothesis on $G$. 

\vskip2mm\noindent 
\textbf{Case $3$.} $G$ has more than one minimal normal subgroups, each isomorphic to either $\mathrm{SL}_2(4)$ or $\mathrm{SL}_2(8)$. 

\noindent Let $\{K_i: i=1,2,\ldots, r\}$ be all the minimal normal subgroups of $G$. Let $L=K_1\times \cdots \times K_r$ (and note that $r\ge 2$).  We show that $L=G$. If possible, let $L< G$.

\vskip2mm\noindent 
\textbf{Sub-case 3.1} Suppose conjugation by  $G$ preserves conjugacy classes of each $K_i$. Take $x\in G\setminus L$.  Since the automorphisms of $\mathrm{SL}_2(4)(\cong A_5)$ (or of $\mathrm{SL}_2(8)$), preserving its all conjugacy classes are inner, there exists $a_i\in K_i$ such that $xtx^{-1}=a_ita_i^{-1}$ for all $t\in K_i$, so $a_i^{-1}x$ centralizes $K_i$ for all $i$. 

Consider $a_1^{-1}\cdots a_r^{-1}x$, it is in $G\setminus L$. Note that $a_ia_j=a_ja_i$ for $i\neq j$. Since $a_2,\ldots, a_r$ centralize $K_1$ and $a_1^{-1}x$ also centralizes $K_1$, it follows that 
$a_1^{-1}\cdots a_r^{-1}x=(a_2^{-1}\cdots a_r^{-1})a_1^{-1}x$ centralizes $K_1$. Similarly, we see that $a_1^{-1}\cdots a_r^{-1}x$ centralizes $K_2,\ldots, K_r$ and hence $L$. Thus, $C_G(L)\neq 1$ and since $L$ is non-abelian, $C_G(L)\neq G$. Thus, $C_G(L)$ contains a minimal normal subgroup of $G$, and it must be $K_i$ for some $i\in \{1,\ldots, r\}$. But $K_i$ does not centralize $K_i$, a contradiction. 

\vskip2mm\noindent 
\textbf{Sub-case 3.2}
Suppose conjugation by $G$ does not preserve conjugacy classes of some $K_i$, i.e. there is $x\in G$, which acts on $K_i$ by a \textit{non-inner} automorphism of $K_i$. If $
\phi:G\rightarrow \mathrm{Aut}(K_i)$ is the $G$-action on $K_i$, then $\phi(K_i)=\mathrm{Inn}(K_i)\cong K_i$, $\phi(x)\notin \mathrm{Inn}(K_i)$, and $[\mathrm{Aut}(K_i):\mathrm{Inn}(K_i)]\in\{2,3\}$. It follows that $\phi(G)=\mathrm{Aut}(K_i) \cong S_5$ or $\mathrm{Ree}(3)$. But then, as seen in Case $1$, $\phi(G)$ (hence $G$) contains a vanishing element of prime power order with conjugacy class size divisible by three distinct primes, a contradiction. 

Thus, we must have $G=L$, i.e. $G \cong K_1\times \cdots \times K_r$, where all $K_i$ are isomorphic to either $\mathrm{SL}_2(4)$
 or $\mathrm{SL}_2(8)$. This completes the proof of Theorem \ref{thmb}.
\hfill$\square$

%
%

%
%
\vskip2mm\noindent 
\textsuperscript{\textdagger}\textsc{Sonakshee Arora}, Department of Mathematics, Indian Institute of Technology Jammu, Jagti, NH-$44$, PO Nagrota, Jammu-$181221$, J$\&$K,  India. 
Email: \texttt{\color{blue}sonakshee.arora@iitjammu.ac.in}
\vskip2mm\noindent 
\textsuperscript{$\ddagger$}\textsc{Rahul Dattatraya Kitture}, Department of Mathematics, Indian Institute of Technology Jammu, Jagti, NH-$44$, PO Nagrota, Jammu-$181221$, J$\&$K, India. Email: \texttt{\color{blue}rahul.kitture@iitjammu.ac.in}
%
%
\end{document}